\documentclass[11pt]{amsart}
\usepackage{amsmath,amssymb,amsthm,mathrsfs}
\usepackage[all]{xy}
\usepackage{color}

\newtheorem{Def}{Definition}[section]
\newtheorem{Thm}[Def]{Theorem}
\newtheorem{Prop}[Def]{Proposition}

\newtheorem{Lemma}[Def]{Lemma}
\newtheorem{Conj}[Def]{Conjecture}
\newtheorem{Rmk}[Def]{Remark}

\newcommand{\ra}{\rightarrow}
\newcommand{\C}{\mathbb{C}}
\newcommand{\PP}{\mathbb{P}}
\newcommand{\R}{\mathbb{R}}
\newcommand{\Z}{\mathbb{Z}}
\newcommand{\Q}{\mathbb{Q}}

\newcommand{\HH}{\mathrm{HH}}
\newcommand{\D}{\mathcal{D}}
\newcommand{\Hom}{\mathrm{Hom}}
\newcommand{\id}{\mathrm{id}}
\newcommand{\HMF}{\mathrm{HMF^{\mathrm{gr}}}}

\newcommand{\T}{\mathrm{T}}
\newcommand{\cO}{\mathcal{O}}

\begin{document}
\title{Entropy of an autoequivalence on Calabi--Yau manifolds}
\author{Yu-Wei Fan}
\date{}

\maketitle

\begin{abstract}
We prove that the categorical entropy of the autoequivalence $\T_{\cO}\circ(-\otimes\cO(-1))$
on a Calabi--Yau manifold of dimension $d\geq3$ is the unique positive real number $\lambda$ satisfying
$$
\sum_{k\geq 1}\frac{\chi(\cO(k))}{e^{k\lambda}}=e^{(d-1)t}.
$$
We then use this result to construct the first counterexamples of a conjecture on categorical entropy by Kikuta and Takahashi.
\end{abstract}

\section{Introduction}
In the pioneering work \cite{DHKK}, Dimitrov--Haiden--Katzarkov--Kontsevich introduced 
the notion of categorical entropy of an endofunctor on a triangulated category with a split generator.
A typical example of such triangulated category is given by the bounded derived category $\D^b(X)$ of coherent sheaves on a variety $X$. When $X$ is a smooth projective variety with ample (anti-)canonical bundle,
the group of autoequivalences on $\D^b(X)$ is well-understood \cite{BO}.
It is generated by tensoring line bundles, automorphisms on the variety, and degree shifts.
The categorical entropy in this case was computed by Kikuta and Takahashi \cite{KT}.
On the other hand, when the variety is Calabi--Yau, there are much more autoequivalences on $\D^b(X)$ because of the presence of spherical objects.

In this article, we show that the composite of the simplest spherical twist $\T_{\cO_X}$ with $-\otimes\cO(-H)$ already gives an interesting categorical entropy.

\begin{Thm}[= Theorem \ref{Thm:A}]
Let $X$ be a strict Calabi--Yau manifold over $\C$ of dimension $d\geq 3$.
Consider the autoequivalence $\Phi:=\T_{\cO_X}\circ(-\otimes\cO(-H))$ on $\D^b(X)$.
The categorical entropy $h_t(\Phi)$ is a positive function in $t\in\R$.
Moreover, for any $t\in\R$, $h_t(\Phi)$ is the unique $\lambda>0$
satisfying
$$
\sum_{k\geq 1}\frac{\chi(\cO(kH))}{e^{k\lambda}}=e^{(d-1)t}.
$$
\end{Thm}

A simple argument on Hilbert polynomial shows that this equation defines an algebraic curve over $\Q$ in the coordinate $(e^t,e^{\lambda})$.
Thus the algebraicity conjecture in \cite[Question 4.1]{DHKK} holds in this case.

The notion of categorical entropy is a categorical analogue of topological entropy of a continuous surjective self-map on a compact metric space.
There is a fundamental theorem of topological entropy due to Gromov and Yomdin.

\begin{Thm}[\cite{Gro1,Gro2,Yom}]\label{Thm:GY}
Let $M$ be a compact K\"ahler manifold and let $f:M\ra M$ be a surjective holomorphic map. Then
$$
h_{\mathrm{top}}(f)=\log\rho(f^*),
$$
where $h_{\mathrm{top}}(f)$ is the topological entropy of $f$, and $\rho(f^*)$ is the spectral radius of $f^*:H^*(M;\C)\ra H^*(M;\C)$.
\end{Thm}

Kikuta and Takahashi proposed the following analogous conjecture on categorical entropy.

\begin{Conj}[\cite{KT}]\label{Conj:KT}
Let $X$ be a smooth proper variety over $\C$ and let $\Phi$ be an autoequivalence on 
$\D^b(X)$. Then
$$
h_0(\Phi)=\log\rho(\HH_{\bullet}(\Phi)),
$$
where $\HH_{\bullet}(\Phi):\HH_{\bullet}(X)\ra\HH_{\bullet}(X)$ is the induced $\C$-linear isomorphism on the Hochschild homology group of $X$, and $\rho(\HH_{\bullet}(\Phi))$
is its spectral radius.
\end{Conj}

Note that one can replace $\HH_{\bullet}(\Phi)$ by the induced Fourier-Mukai type action on the cohomology
$\Phi_{H^*}:H^*(X;\C)\rightarrow H^*(X;\C)$,
because there is a commutative diagram
    \begin{equation*}
        \xymatrix{\HH_{\bullet}(X)\ar[rr]^{\HH_{\bullet}(\Phi)}\ar[d]_{I_K^{X}}&&\HH_{\bullet}(X)\ar[d]^{I_K^{X}}\\ H^*(X;\C)\ar[rr]^{\Phi_{H^*}}&&H^*(X;\C),}
    \end{equation*}
where $I_K^{X}$ is the modified Hochschild--Kostant--Rosenberg isomorphism \cite[Theorem 1.2]{MSM}.
Hence $\rho(\HH_{\bullet}(\Phi))=\rho(\Phi_{H^*})$.

Conjecture \ref{Conj:KT} has been proved in several cases. See \cite{K,KST,KT,Yos}.

Now we explain the motivation of the present work, namely why we should not expect
Conjecture \ref{Conj:KT} to hold in general.
We first note that Theorem \ref{Thm:GY} does not hold if $f$ is not holomorphic.
For example, there is a construction by Thurston \cite{Thur} of pseudo-Anosov maps that act trivially on the cohomology, and thus have zero spectral radius.
Moreover, Dimitrov--Haiden--Katzarkov--Kontsevich \cite[Theorem 2.18]{DHKK} showed that the categorical entropy of the induced autoequivalence on the derived Fukaya category is equal to $\log\lambda$, where $\lambda>1$ is the stretch factor of the pseudo-Anosov map.
Hence the analogous statement of Conjecture \ref{Conj:KT} is not true if $\D^b(X)$ is replaced by the derived Fukaya categories of the symplectic manifolds considered in \cite{DHKK,Thur}.

Motivated by homological mirror symmetry, one may expect to find counterexamples of Conjecture \ref{Conj:KT} on the derived categories of coherent sheaves on Calabi--Yau manifolds.
In other words, the discrepancy between complex and symplectic geometry should lead to the discrepancy between Theorem \ref{Thm:GY} and Conjecture \ref{Conj:KT}.

Using Theorem 1.1, we construct counterexamples of Conjecture \ref{Conj:KT}.

\begin{Prop}[= Proposition \ref{Prop:B}]
For any even integer $d\geq4$,
let $X$ be a Calabi--Yau hypersurface in $\C\PP^{d+1}$ of degree $(d+2)$ and $\Phi=\T_{\cO_X}\circ(-\otimes\cO(-1))$. Then
$$
h_0(\Phi)> 0=\log\rho(\Phi_{H^*}).
$$
In particular, Conjecture \ref{Conj:KT} fails in this case.
\end{Prop}

Interestingly, as pointed out to the author by Genki Ouchi, the same autoequivalence does not produce counterexamples of Conjecture \ref{Conj:KT} if $X$ is an \emph{odd} dimensional Calabi--Yau manifold (see Remark \ref{Rmk:Genki}).

After the present paper was posted online, Ouchi also proved a result similar to Proposition \ref{Prop:B} for K3 surfaces \cite{Ouchi}.

\section{Preliminaries}

\subsection{Categorical entropy}

We recall the notion of categorical entropy introduced by Dimitrov-Haiden-Katzarkov-Kontsevich \cite{DHKK}.

Let $\D$ be a triangulated category.
A triangulated subcategory is called \emph{thick} if it is closed under taking direct summands.
The \emph{split closure} of an object $E\in\D$ is the smallest thick triangulated subcategory containing $E$.
An object $G\in\D$ is called a \emph{split generator} if its split closure is $\D$.

\begin{Def}[\cite{DHKK}]
Let $E$ and $F$ be non-zero objects in $\D$.
If $F$ is in the split closure of $E$, then the \emph{complexity} of $F$ relative to $E$ is defined to be the function
$$
\delta_t(E,F):=
\inf\left\{ \displaystyle\sum_{i=1}^k e^{n_i t} \,\middle|\,
\begin{xy}
(0,5) *{0}="0", (20,5)*{A_{1}}="1", (30,5)*{\dots}, (40,5)*{A_{k-1}}="k-1", 
(60,5)*{F\oplus F'}="k",
(10,-5)*{E[n_{1}]}="n1", (30,-5)*{\dots}, (50,-5)*{E[n_{k}]}="nk",
\ar "0"; "1"
\ar "1"; "n1"
\ar@{-->} "n1";"0" 
\ar "k-1"; "k" 
\ar "k"; "nk"
\ar@{-->} "nk";"k-1"
\end{xy}
\, \right\},
$$
where $t$ is a real parameter that keeps track of the shiftings.
\end{Def}

Note that when the category $\D$ is $\Z_2$-graded in the sense that
$[2]\cong\mathrm{id}_\D$,
only the value at $t$ equals to zero, $\delta_0(E,F)$, will be of any use.

\begin{Def}[\cite{DHKK}]
Let $D$ be a triangulated category with a split generator $G$
and let $\Phi:\D\ra\D$ be an endofunctor.
The \emph{categorical entropy} of $\Phi$ is defined to be the function $h_t(\Phi):\R\ra[-\infty,\infty)$
given by
$$
h_t(\Phi):=\lim_{n\ra\infty}\frac{1}{n}\log\delta_t(G,\Phi^nG).
$$
\end{Def}

\begin{Lemma}[\cite{DHKK}]
The limit $\lim_{n\ra\infty}\frac{1}{n}\log\delta_t(G,\Phi^nG)$ exists in $[-\infty,\infty)$ for every $t\in\R$, and is independent of the choice of the split generator $G$.
\end{Lemma}

We will use the following proposition to compute categorical entropy.

\begin{Prop}[\cite{DHKK,KT}]\label{Prop:DHKK}
Let $G$ and $G'$ be split generators of $\D^b(X)$
and let $\Phi$ be an autoequivalence on $\D^b(X)$.
Then the categorical entropy equals to
$$
h_t(\Phi)=\lim_{n\ra\infty}\frac{1}{n}\log\sum_{l\in\Z}\dim\Hom^l_{\D^b(X)}(G,\Phi^nG')e^{-lt}.
$$
\end{Prop}

\subsection{Spherical objects and spherical twists}
We recall the notions of spherical objects and spherical twists introduced by Seidel-Thomas \cite{ST}.
They are the categorical analogue of Lagrangian spheres in symplectic manifolds and the Dehn twists along Lagrangian spheres.

\begin{Def}[\cite{ST}]
An object $S\in\D^b(X)$ is called \emph{spherical} if $S\otimes\omega_X\cong S$ and
$\Hom^{\bullet}_{\D^b(X)}(S,S)=\C\oplus\C[-\dim X]$.
\end{Def}

\begin{Def}[\cite{ST}]
The \emph{spherical twist} $\T_S$ with respect to a spherical object $S$ is an autoequivalence on $\D^b(X)$ given by
$$
E\mapsto\T_S(E):=\mathrm{Cone}(\Hom^{\bullet}_{\D^b(X)}(S,E)\otimes S\ra E).
$$
\end{Def}

We also recall the definition of strict Calabi--Yau manifolds.

\begin{Def}\label{Def:Strict CY}
A smooth projective variety $X$ is called \emph{strict Calabi--Yau}
if $\omega_X\cong\cO_X$ and $H^i(X,\cO_X)=0$ for all $0<i<\dim X$.
This is equivalent to the condition that $\cO_X$ is a spherical object.
\end{Def}

\section{Computation of categorical entropy}

We fix the notations and assumptions that will be used throughout this section.
We work over complex number field $\C$.

\subsection*{Notations}
\begin{itemize}
\item $X$ is a strict Calabi--Yau manifold (Definition \ref{Def:Strict CY}) with a very ample line bundle $H$.
\item $d:=\mathrm{dim}_{\C}X\geq3$.
\item $\cO:=\cO_X$ and $\cO(k):=\cO_X(kH)$.
\item $a_k:=h^0(\cO(k))=\chi(\cO(k))$ for $k>0$.
\item $G:=\oplus_{i=1}^{d+1}\cO(i)$ and $G':=\oplus_{i=1}^{d+1}\cO(-i)$.
By a result of Orlov \cite{Orlov}, both $G$ and $G'$ are split generators of $\D^b(X)$.
\end{itemize}

The goal of this section is to prove the following theorem.

\begin{Thm}\label{Thm:A}
Let $X$ be a strict Calabi--Yau manifold over $\C$ of dimension $d\geq 3$.
Consider the autoequivalence $\Phi:=\T_{\cO}\circ(-\otimes\cO(-1))$ on $\D^b(X)$.
The categorical entropy $h_t(\Phi)$ is a positive function in $t\in\R$.
Moreover, for any $t\in\R$, $h_t(\Phi)$ is the unique $\lambda>0$
satisfying
$$
\sum_{k\geq 1}\frac{\chi(\cO(k))}{e^{k\lambda}}=e^{(d-1)t}.
$$
\end{Thm}

We begin with a lemma that will be crucial in the computation of categorical entropy.

\begin{Lemma}\label{Lemma:degree}
	\begin{enumerate}
	\item For any integers $n\geq 0$ and $k>0$,
$\Hom^{l}(\cO,\Phi^n(G')\otimes\cO(-k))$ is zero except for $l=d,d+(d-1),\ldots,d+n(d-1)$.
	\item For $n\geq0$ and $k>0$, define
$$
B_{n,k}:=\sum_{m=0}^n\dim\Hom^{d+m(d-1)}(\cO,\Phi^n(G')\otimes\cO(-k))
\cdot e^{-m(d-1)t}.
$$
In particular, $B_{0,k}=\dim\Hom^d(\cO,\cO(-k))=a_k$.
There is a recursive relation among $B_{n,k}$'s:
\begin{equation}\label{Recursion}
B_{n,k}=B_{n-1,k+1}+a_ke^{-(d-1)t}B_{n-1,1}.  \tag{*}
\end{equation}
	\end{enumerate}
\end{Lemma}

\begin{proof}
We prove the first assertion by induction on $n$.
The statement is true when $n=0$ by Kodaira vanishing theorem and Serre duality.

By the definition of $\Phi$, there is an exact triangle:
$$
\Phi^{n-1}(G')\otimes\cO(-1)\rightarrow\Phi^n(G')\rightarrow
\Hom^{\bullet}(\cO,\Phi^{n-1}(G')\otimes\cO(-1))\otimes\cO[1]\xrightarrow{+1}.
$$

By tensoring it with $\cO(-k)$ and applying $\Hom^{\bullet}(\cO,-)$, we get a long exact sequence:
\begin{multline}
\Hom^{\bullet}(\cO,\Phi^{n-1}(G')\otimes\cO(-k-1))\rightarrow\Hom^{\bullet}(\cO,\Phi^n(G')\otimes\cO(-k))
\\
\rightarrow\Hom^{\bullet}(\cO,\Phi^{n-1}(G')\otimes\cO(-1))\otimes\Hom^{\bullet}(\cO,\cO(-k))[1]
\xrightarrow{+1}\cdots \notag
\end{multline}

Suppose the statement is true for $n-1$. Then the first complex in the long exact sequence is non-zero only at degree $d,d+(d-1),\ldots,d+(n-1)(d-1)$, and the third complex is non-zero only at degree $d+(d-1),d+2(d-1),\ldots,d+n(d-1)$.
Since we assume the dimension $d\geq 3$, the long exact sequence splits into short exact sequences, and the proof follows.

The recursive relation in the second assertion also follows from the above long exact sequence.
\end{proof}

%
%
%


\begin{Lemma}\label{Lemma:recursion}
Define
$$
P_{s,k}:=\sum_{\substack{i_1+\cdots+i_q=s\\i_1\geq k}} a_{i_1}a_{i_2}\cdots a_{i_q}
e^{-q(d-1)t},
$$
where the summation runs over all ordered partitions of $s$ with the first piece no less than $k$. Then we have
$$
B_{n,k}=a_{n+k}+\sum_{s=1}^n a_s P_{n+k-s,k}.
$$
\end{Lemma}

\begin{proof}
Induction on $n$  and use the recursive relation (\ref{Recursion}).
\end{proof}

Now we are ready to prove Theorem \ref{Thm:A}.

\begin{proof}[Proof of Theorem \ref{Thm:A}]
By Proposition \ref{Prop:DHKK}, the categorical entropy can be written as
\begin{equation*}
\begin{split}
h_t(\Phi) & = \lim_{n\ra\infty}\frac{1}{n}\log\sum_l \dim\Hom^l (G,\Phi^n(G'))e^{-lt} \\
 & =  \lim_{n\ra\infty}\frac{1}{n}\log\sum_{k=1}^{d+1}\sum_l
\dim\Hom^l(\cO,\Phi^n(G')\otimes\cO(-k))e^{-lt} \\
 & = \lim_{n\ra\infty}\frac{1}{n}\log\sum_{k=1}^{d+1} B_{n,k}.
\end{split}
\end{equation*}

By Lemma \ref{Lemma:recursion} and the fact that $\{a_k\}$ is an increasing sequence, 
we have $B_{n,1}\leq B_{n,k+1}\leq B_{n+k,1}$. Hence
$$
h_t(\Phi)= \lim_{n\ra\infty}\frac{1}{n}\log  B_{n,1}.
$$

Define $C_n:=B_{n,1}e^{-(d-1)t}$ and $a'_k:=a_ke^{-(d-1)t}$. Then again by Lemma \ref{Lemma:recursion},
\begin{equation*}
\begin{split}
C_n & = \Big(a_{n+1}+\sum_{s=1}^na_sP_{n+1-s,1}\Big)e^{-(d-1)t} \\
  & =\sum_{i_1+\cdots+i_q=n+1}a'_{i_1}a'_{i_2}\cdots a'_{i_q}.
\end{split}
\end{equation*}

Thus
\begin{equation*}
C_n=a'_1C_{n-1}+a'_2C_{n-2}\cdots+a'_nC_0+a'_{n+1}.
\end{equation*}

Notice that $h_t(\Phi)= \lim_{n\ra\infty}\frac{1}{n}\log  B_{n,1}=\lim_{n\ra\infty}\frac{1}{n}\log C_n$ is a positive real number for any $t\in\R$, because there exists some $k>0$ (depending on $d$ and $t$) such that $a'_k=a_ke^{-(d-1)t}>1$, hence
$$
\lim_{n\ra\infty}\frac{1}{n}\log C_n  \geq\lim_{n\ra\infty}\frac{1}{n}\log (a'_k)^
{\lfloor\frac{n+1}{k}\rfloor} 
 =\frac{1}{k}\log a'_k >0.
$$

We claim that the categorical entropy $h_t(\Phi)$ is the unique $\lambda>0$
satisfying
$$
\sum_{k\geq1}\frac{a'_k}{e^{k\lambda}}=1.
$$
Or equivalently,
$$
\sum_{k\geq 1}\frac{\chi(\cO(k))}{e^{k\lambda}}=e^{(d-1)t}.
$$
This will conclude the proof of the theorem.
We prove the claim in the following lemma,
with a slight change in notations.

\end{proof}

\begin{Lemma}
Let $\{a_n\}_{n\geq1}$ be a sequence of positive real numbers, and define
$\{C_n\}_{n\geq0}$ inductively by
\begin{equation}\label{Cn}
C_n=a_1C_{n-1}+a_2C_{n-2}\cdots+a_nC_0+a_{n+1}. \tag{**}
\end{equation}
Suppose that 
$$
\lim_{n\rightarrow\infty}\frac{1}{n}\log C_n=\lambda>0,
$$
then we have
$$
\sum_{k\geq1}\frac{a_k}{e^{k\lambda}}=1.
$$
\end{Lemma}

\begin{proof}
Assume that
$
\sum_{k\geq1}\frac{a_k}{e^{k\lambda}}>1.
$
Then there exists some $m>0$ so that
$
\sum_{k=1}^m\frac{a_k}{e^{k\lambda}}>1.
$
Moreover, there exists some $\lambda_1>\lambda$ such that
$
\sum_{k=1}^m\frac{a_k}{e^{k\lambda}}>\sum_{k=1}^m\frac{a_k}{e^{k\lambda_1}}>1.
$

Choose a constant $D_1>0$ such that $C_k>D_1e^{(k+1)\lambda_1}$
for all $0\leq k\leq m$.
We can then prove by induction that $C_n>D_1e^{(n+1)\lambda_1}$ holds for all $n\geq0$.
Indeed, by the induction hypothesis, (\ref{Cn}) and
$\sum_{k=1}^m\frac{a_k}{e^{k\lambda_1}}>1$, for $n>m$, we have
\begin{equation*}
\begin{split}
C_n & >a_1C_{n-1}+\cdots+a_mC_{n-m} \\
  & > D_1(a_1e^{n\lambda_1}+\cdots+a_me^{(n-m+1)\lambda_1}) \\
  & > D_1e^{(n+1)\lambda_1}.
\end{split}
\end{equation*}

However, this contradicts with the assumption that
$
\lim_{n\rightarrow\infty}\frac{1}{n}\log C_n=\lambda
$
since $\lambda_1>\lambda$.
Hence we have
$
\sum_{k\geq1}\frac{a_k}{e^{k\lambda}}\leq1.
$

On the other hand, assume that 
$
\sum_{k\geq1}\frac{a_k}{e^{k\lambda}}<1.
$
Then there exists some $0<\lambda_2<\lambda$ such that
$
\sum_{k\geq1}\frac{a_k}{e^{k\lambda}}<\sum_{k\geq1}\frac{a_k}{e^{k\lambda_2}}<1.
$

Choose a constant $D_2>1$ such that $C_0<D_2e^{\lambda_2}$.
We can prove by induction that $C_n<D_2e^{(n+1)\lambda_2}$ for all $n\geq0$.
Indeed, by induction hypothesis, (\ref{Cn}) and
$\sum_{k\geq1}\frac{a_k}{e^{k\lambda_2}}<1$, we have
$$
C_n  < D_2(a_1e^{n\lambda_2}+\cdots+a_ne^{\lambda_2}+a_{n+1}) 
   < D_2 e^{(n+1)\lambda_2}.
$$

This again contradicts with the assumption that
$
\lim_{n\rightarrow\infty}\frac{1}{n}\log C_n=\lambda
$
since $\lambda_2<\lambda$.

This concludes the proof of the lemma.
\end{proof}

\section{Counterexample of Kikuta-Takahashi}
Using Theorem \ref{Thm:A}, we can now construct counterexamples of Conjecture~\ref{Conj:KT}.

\begin{Prop}\label{Prop:B}
For any even integer $d\geq4$,
let $X$ be a Calabi--Yau hypersurface in $\C\PP^{d+1}$ of degree $(d+2)$ and $\Phi=\T_{\cO}\circ(-\otimes\cO(-1))$. Then
$$
h_0(\Phi)> 0=\log\rho(\Phi_{H^*}).
$$
In particular, Conjecture \ref{Conj:KT} fails in this case.
\end{Prop}

\begin{proof}
By Theorem \ref{Thm:A}, we only need to show that the spectral radius $\rho(\Phi_{H^*})$ equals to 1.

Consider another autoequivalence $\Phi':=\T_{\cO}\circ(-\otimes\cO(1))$ on $\D^b(X)$.
By \cite[Proposition 5.8]{BFK}, there is a commutative diagram
    \begin{equation*}
        \xymatrix{\D^b(X)\ar[rr]^{\Phi'}\ar[d]_{\Psi}&&\D^b(X)\ar[d]^{\Psi}\\ 
	\HMF(W)\ar[rr]^{\tau}&&\HMF(W).}
    \end{equation*}
Here $W$ is the defining polynomial of $X$,
$\HMF(W)$ is the associated graded matrix factorization category,
$\Psi$ is an equivalence introduced by Orlov \cite{Orl2},
and $\tau$ is the grade shift functor on $\HMF(W)$
which satisfies $\tau^{d+2}=[2]$.
Hence we have $(\Phi')^{d+2}=[2]$ and $(\Phi')_{H^*}^{d+2}=\id_{H^*}$.

On the other hand, $(\T_{\cO})_{H^*}$ is an involution on $H^*(X;\C)$ when $X$ is an even dimensional strict Calabi--Yau manifold (\cite[Corollary 8.13]{Huy}).
Thus $(\T_{\cO})_{H^*}=(\T_{\cO})_{H^*}^{-1}$.
Hence we also have $\Phi_{H^*}^{d+2}=\id_{H^*}$, which implies that $\rho(\Phi_{H^*})=1$.
\end{proof}

\begin{Rmk}
The autoequivalence $\Phi'$ that we considered in the proof is the one that corresponds to the monodromy around the Gepner point ($\Z_{d+2}$-orbifold point) in the K\"ahler moduli of $X$.
\end{Rmk}

\begin{Rmk}[Ouchi]\label{Rmk:Genki}
The functor $\Phi=\T_{\cO}\circ(-\otimes\cO(-1))$ does not produce counterexamples of Conjecture \ref{Conj:KT} if $X$ is an \emph{odd} dimensional Calabi--Yau manifold.
When $X$ is of odd dimension, Lemma \ref{Lemma:degree} implies
$$
h_0(\Phi)=\lim_{n\ra\infty}\frac{1}{n}\log\chi(G,\Phi^n(G'))\leq\log\rho([\Phi]).
$$
On the other hand, we have $h_0(\Phi)\geq\log\rho([\Phi])$ by Kikuta-Shiraishi-Takahashi \cite[Theorem 2.13]{KST}.
\end{Rmk}

\subsection*{Acknowledgement}
I would like to thank Fabian~Haiden for suggesting this problem to me,
and Genki~Ouchi for reading and pointing out Remark \ref{Rmk:Genki}.
I would also like to thank Philip~Engel, Hansol~Hong, Atsushi~Kanazawa, Koji~Shimizu, Yukinobu~Toda and Cheng-Chiang~Tsai for helpful conversations and correspondences.
Finally, I would like to thank Shing-Tung~Yau and Harvard University Math Department for warm support,
and anonymous referees for suggesting several improvements.
This work is partially supported by Simons Collaboration Grant on Homological Mirror Symmetry.

\bigskip

\noindent Department of Mathematics, Harvard University, 
One Oxford Street, Cambridge, MA 02138, USA.

\smallskip

\noindent email: {\tt ywfan@math.harvard.edu}

\end{document}